\documentclass[12pt,reqno]{amsart}
\usepackage{dsfont}
\usepackage{bbm}
\usepackage{mathrsfs}
\usepackage{color}
\usepackage{comment}
\usepackage{caption}
\usepackage{enumitem}
\usepackage{diagbox}
\allowdisplaybreaks[4]
\usepackage{amssymb, amsmath}
\usepackage{bm}
\usepackage{verbatim}
\usepackage{graphicx,amsmath,amsfonts,amssymb,amscd,amsthm}
\usepackage{hyperref}
\usepackage[numbers,sort&compress]{natbib}
\usepackage{enumitem}
\usepackage{xstring}   
\usepackage{etoolbox}
\usepackage{multirow} 
\usepackage[numbers,sort&compress]{natbib}
\usepackage{array}
\usepackage{booktabs}
\usepackage{multirow}
\usepackage{makecell}
\usepackage{abstract}
\usepackage{float}
\usepackage{pifont}
\setenumerate[1]{itemsep=1pt,partopsep=0pt,parsep=\parskip,topsep=1pt}
\setitemize[1]{itemsep=1pt,partopsep=0pt,parsep=\parskip,topsep=1pt}

\newtheorem{theorem}{Theorem}[section]
\newtheorem{lemma}{Lemma}[section]
\newtheorem{proposition}{Proposition}[section]

\newtheorem{definition}{Definition}[section]

 \newcommand{\Bp}{\begin{proof}}
 \newcommand{\Ep}{\end{proof}}
\hypersetup{colorlinks,linkcolor=blue,citecolor=blue}
\newtheoremstyle{kai}
{3pt} {3pt} {} {} {\bfseries} {.} {.5em} {}
\def\EquationsBySection{\def\theequation
{\thesection.\arabic{equation}}
\@addtoreset{equation}{section}}

\newcommand\old[1]{}
\renewcommand{\theequation}{\thesection.\arabic{equation}}

 \newcommand{\R}{\mathbb{R}}
\newcommand{\V}{\mathrm{d} \mathrm{V_g}}
\newcommand{\red}{\textcolor[rgb]{1.00,0.00,0.00}}

\newcommand{\grad}{\nabla_M}

\newcommand{\beq}{\begin{equation}}
\newcommand{\eeq}{\end{equation}}
\newfloat{figtab}{htb}{fgtb}
\makeatletter
  \newcommand\figcaption{\def\@captype{figure}\caption}
  \newcommand\tabcaption{\def\@captype{table}\caption}
\makeatother
\DeclareMathOperator{\dive}{div}
\numberwithin{equation}{section}
\addcontentsline{toc}{section}{Acknowledgement}

\begin{document}
\author{Xin Xu}
\address{Xin Xu, School of Mathematical Sciences, South China Normal University, Guangdong 510631, China.}
\email{xuxin1994pkq@163.com}

\author{Kexin Zhang}
\address{Kexin Zhang, School of Mathematics, Sun Yat-sen University, Guangzhou 510275, China.}
\email{kxzmath@163.com}

\title
{Asymptotics of the principal eigenvalue of an elliptic operator on closed and orientable  Riemannian manifolds: small diffusion}
\date{}
\maketitle

\begin{abstract}
This paper is concerned with the asymptotic behavior of the principal eigenvalue $\lambda(D)$ of the elliptic eigenvalue problem
\[
-D\Delta_{M}u - a\langle \nabla_M f, \nabla_M u\rangle_g + c u = \lambda(D)u,
\]
posed on a closed orientable Riemannian manifold $(M,g)$, in the small-diffusion limit $D \to 0^+$. Under the assumption that $f$ is a Morse function on $M$, we establish that the limiting value $\lim_{D\to 0}\lambda(D)$ is completely characterized by the critical points of $f$ and the associated Riemannian Hessian, specifically through the values of $c$ and the  Riemannian Hessian eigenvalues at those points.
\end{abstract}

\noindent \textbf{Keywords:} Principal eigenvalue, asymptotic behavior,  Riemannian manifold.
\vskip20pt

\section{Introduction}

In this paper, we investigate the asymptotic behavior of the principal eigenvalue $\lambda(D)$ (for $D>0$) of the elliptic eigenvalue problem
 \begin{equation}\label{E}
  -D\Delta_{M}u-a\langle  \nabla_M f, \nabla_M u\rangle_g +c u=\lambda(D)u, \text{on } M, 
\end{equation}
in the small-diffusion limit $D \to 0^+$. Here, $(M,g)$ denotes an $N$-dimensional closed (i.e., compact and boundaryless) orientable Riemannian manifold, $a$ is a fixed positive constant, and $f, c: M \to \mathbb{R}$ are smooth functions. The operators $\nabla_M$ and $\Delta_M$ stand for the gradient and the Laplace--Beltrami operator induced by the metric $g$, respectively.

The drifted Laplacian $\Delta_f=\Delta_{M}+\langle  \nabla_M f, \nabla_M \cdot\rangle_g$ arises naturally in a variety of mathematical contexts \cite{zhoudetang,book,CCL6,CCL7,GT}. Such operators enjoy favorable analytic properties: although they are not self-adjoint with respect to the standard Riemannian volume measure, they become self-adjoint when formulated on appropriate weighted Sobolev spaces (see, e.g., \cite{zhoudetang}). This observation provides a variational framework for studying the principal eigenvalue. Within this framework, combined with classical elliptic theory, it follows that problem \eqref{E} admits a unique principal eigenvalue $\lambda(D)$. Moreover, $\lambda(D)$ is real and simple, and its corresponding principal eigenfunction   can be chosen to be strictly positive on $M$ (see, e.g.  \cite{RGA}).

\subsection{Main result}

Our primary goal is to characterize the asymptotic behavior of the principal eigenvalue $\lambda(D)$ in the small-diffusion limit $D \to 0^+$. To this end, we introduce the critical set of $f$:
\[
\Sigma = \bigl\{ p \in M \bigm| |\nabla_M f(p)| = 0 \bigr\}.
\]
Our main theorem is then stated as follows.

\begin{theorem}\label{smalld}
Assume that $f$ is a Morse function on a closed orientable Riemannian manifold $(M,g)$. Let $\lambda(D)$ denote the principal eigenvalue of problem \eqref{E}. Then
\[
\lim_{D \to 0} \lambda(D)
= \min_{p \in \Sigma} \left\{ c(p) + \frac{a}{2} \sum_{i=1}^N \bigl( |\kappa_i(p)| + \kappa_i(p) \bigr) \right\},
\]
where $\kappa_1(p), \ldots, \kappa_N(p)$ are the eigenvalues of the Riemannian Hessian operator $H_p(f)$: $T_p M\to T_p M$ associated with the metric $g$, defined by 
\[
\operatorname{Hess}_p (f)(v, w) =g(H_p(f)(v),w), 
\quad \forall  v,w\in T_p M.
\]
Equivalently, for each $i = 1, \ldots, N$, there exists a nonzero vector $\zeta_i \in T_p M$ such that
\[H_p(f)(\zeta_i)=\kappa_i(p)\zeta_i.
\]
\end{theorem}

The theorem reveals that the small-diffusion asymptotics of $\lambda(D)$ on closed manifolds are entirely governed by the critical points of $f$ and the associated Riemannian Hessian, with no boundary contributions arising due to the absence of boundary. This result constitutes a Riemannian analogue of the known asymptotic characterization of principal eigenvalues in bounded Euclidean domains \cite{FA7273,DEF73,DF78,FS97,CL12,PZZ19,XZ26}. For the small-diffusion/large-advection  asymptotics of the second order elliptic and time-periodic parabolic operators in the Euclidean case, we refer to \cite{Bcmp,CL,LLPZ,LLPZd,LLPZx,PZhao} and the references therein. To contextualize our findings, we briefly review several representative results in the Euclidean setting.

\subsection{Existing results}
The asymptotic behavior of the principal eigenvalue for drifted Laplacians in the small-diffusion limit $D \to 0^+$ has been extensively studied in the literature. The majority of existing results, however, are confined to bounded domains in Euclidean space $\mathbb{R}^N$, we refer to \cite{CL12,PZZ19}. Specifically, let $\Omega \subset \mathbb{R}^N$ be a bounded domain, and consider the eigenvalue problem
\begin{equation}\label{E_Euclidean}
\left\{
\begin{aligned}
&-D\Delta v - a \nabla m(x)\cdot \nabla v + c(x)v = \lambda_{\Omega}(D)v \quad &&\text{in } \Omega,\\
&\mathcal{B}v = 0 \quad &&\text{on } \partial\Omega,
\end{aligned}
\right.
\end{equation}
subject to, for instance, Dirichlet or Neumann boundary conditions, where $m \in C^2(\overline{\Omega})$ and $c \in C(\overline{\Omega})$.

For notational convenience, we introduce the following critical sets: 
$$\Sigma_{\Omega}^1=\{x\in \Omega\bigm| |\nabla m(x)|=0 \},\quad \Sigma_{\Omega}^2=\{x\in \partial\Omega\bigm| |\nabla m(x)|=0 \},$$ 
and $$\Sigma_{\Omega}^3=\{x\in \partial\Omega\bigm| |\nabla m(x)|=\nabla m(x)\cdot \nu(x)>0 \}
 $$
where $\nu(x)$ denotes the outward unit normal to $\partial\Omega$ at $x$.

Under the Dirichlet boundary condition $\mathcal{B}v=v=0$ on $\partial \Omega$ (Dirichlet boundary condition), Peng, Zhang and Zhou in \cite{PZZ19} established the following results:
\begin{enumerate}
\item[(i)]
If $\Sigma_{\Omega}^1\cup\Sigma_{\Omega}^2=\emptyset$, then
\[
\lim_{D\to0}\lambda_{\Omega}(D)=+\infty.
\]

\item[(ii)]
If $\Sigma_{\Omega}^1\cup\Sigma_{\Omega}^2\neq\emptyset$ and if for every
$x\in\Sigma_{\Omega}^2$, $\nu(x)$ is an eigenvector of $D^2m(x)$ with corresponding eigenvalue $\kappa_N(x)=0$, then
\[
\lim_{D\to0}\lambda_{\Omega}(D)
=
\min_{x\in\Sigma_{\Omega}^1\cup\Sigma_{\Omega}^2}
\left\{
c(x)
+\frac{a}{2}
\sum_{i=1}^{N}
\bigl(|\kappa_i(x)|+\kappa_i(x)\bigr)
\right\},
\]
where $\kappa_1(x),\ \kappa_2(x),\ \ldots,\ \kappa_N(x)$
are the eigenvalues of $D^2m(x)$ with $\kappa_N(x)=0$.
\end{enumerate}

In particular, if $\Sigma_{\Omega}^1\neq \emptyset$ and $\Sigma_{\Omega}^2=\emptyset$, then the formula in (ii) reduces to
\[
\lim_{D\to0}\lambda_{\Omega}(D)
=
\min_{x\in\Sigma_{\Omega}^1}
\left\{
c(x)
+\frac{a}{2}
\sum_{i=1}^{N}
\bigl(|\kappa_i(x)|+\kappa_i(x)\bigr)
\right\}.
\]
We note that this formula is formally identical to that established in our Theorem \ref{smalld} for closed manifolds, with the eigenvalues of the Euclidean Hessian $D^2 m(x)$ replaced by the  eigenvalues of the Riemannian Hessian $\operatorname{Hess}_p f$. 
 
For the Neumann boundary condition, Chen and Lou \cite{CL12}  showed that, besides the interior and boundary critical points of $m$ (namely $\Sigma_{\Omega}^1$ and $\Sigma_{\Omega}^2$), the boundary points in $\Sigma_{\Omega}^3$ also contribute to the asymptotic behavior of the principal eigenvalue. For the Robin boundary condition, Peng, Zhang and Zhou \cite{PZZ19} further established that the limiting behavior of the principal eigenvalue depends not only on the interior and boundary critical points of $m$, but also on the boundary points in $\Sigma_{\Omega}^3$ and the Robin boundary parameter.

The above existing results show that, in bounded Euclidean domains, the small diffusion asymptotic behavior of the principal eigenvalue $\lambda(D)$ are closely related to the critical points  of  $m$. Moreover, the presence of the boundary may introduce additional boundary contributions, whose forms depend on the prescribed boundary condition. In contrast, our present work is concerned with closed Riemannian manifolds, where no boundary effects arise naturally. Our main theorem shows that the Euclidean asymptotic formula corresponding to the interior critical points remains valid in this Riemannian geometric setting.

The proof of Theorem \ref{smalld} is based on the variational characterization of the principal eigenvalue for the drifted Laplacian. The lower bound estimate follows from suitable estimates of the associated energy functional. For the upper bound estimate, we perform a local analysis near each critical point of $f$. By introducing geodesic normal coordinates, the operator is locally reduced to its Euclidean counterpart, while the errors arising from the Riemannian metric are carefully controlled.   

This paper is organized as follows. In Section 2, we present some preliminaries and the varitional characterizations of $\lambda(D)$. In Section 3 and 4, we prove the upper and lower estimates for the principal eigenvalue respectively, which complete the proof of the main theorem. 

\section{Preliminaries}\label{pre}

Let \((M,g)\) be a closed, orientable Riemannian manifold of dimension $N$.

\subsection{Critical points and Morse functions}
Let $(V,\phi)$ be a local coordinate chart of $p^*$ on $M$, and write  
\[
\phi(p^*)=\bigl(y_{1}(p^*),\dots,y_{N}(p^*)\bigr)\triangleq y\in\mathbb{R}^{N},\quad p^*\in V,
\]
where \(y_{k}:V\to\mathbb{R}\) are the associated coordinate functions.
In these coordinates, the Riemannian metric $g$ is represented by the smooth,
symmetric, positive-definite matrix
 \((g_{ij}(y))_{N\times N}\).

\begin{definition}[Critical point \cite{jmlee}]\label{cp}
A point \(p^* \in M\) is a {\bf{critical point}} of a smooth function \(f: M \to \mathbb{R}\) if \((df)_{p^*} = 0\). A critical point \(p^*\) is called {\bf{non-degenerate}} if the Hessian bilinear form \(\operatorname{Hess}_{p^*}(f): T_{p^*}M \times T_{p^*}M \to \mathbb{R}\) is non-degenerate.
\end{definition}

Let \(\tilde{f} := f \circ \phi^{-1}\) denote the local coordinate representation of \(f\). Since \((df)_{p^*} = 0\), we have
\[
\left.\frac{\partial}{\partial y_{i}}\right\vert_{{p^*}} f
=\left.\frac{\partial}{\partial y_{i}}\right\vert_{\phi({p^*})}(f\circ\phi^{-1})
=\frac{\partial \tilde{f}}{\partial y_{i}}(\phi({p^*}))=0,\quad \text{for all }i=1,\dots,N
\]
by the standard chain rule (see, e.g.,~\cite{jmlee}). Consequently, the coordinate expression for the squared norm of the gradient of \(f\) at \(p^*\) is
\[
|\nabla_M f_{p^*}|^2=g^{ij}\frac{\partial \tilde{f}}{\partial y_{i}}\frac{\partial \tilde{f}}{\partial y_{j}}(\phi({p^*})), \quad \text{(summation over }i,j \text{ by Einstein convention),}
\]
where \(g^{ij}\) is the \((i,j)\)-component of the inverse matrix of the corresponding Riemannian metric \((g_{ij}(y))_{N\times N}\). Hence \((df)_{p^*} = 0\) is equivalent to \((\nabla_M f)_{p^*} = 0\).

Since $p^*$ is a critical point of $f$, the coordinate components of
the Riemannian Hessian (as a symmetric bilinear form, see \cite{morse2024}) at $p^*$ satisfy
\[
\operatorname{Hess}_{p^*}(f)(v,w)
=
\frac{\partial^2\tilde f}
{\partial y_i\partial y_j}
\left(\phi(p^*)\right)v_iw_j,
\]
where $v,w\in T_{p^*}V$, $(v_1,...,v_n)=(d\phi)_{p^*}(v)$ and $(w_1,...,w_n)=(d\phi)_{p^*}(w)$.
Thus, \(p^*\) is non-degenerate if and only if the above bilinear form is non-degenerate, i.e. the matrix
\[
\left(\frac{\partial^2 \tilde f}{\partial y_i \partial y_j}\bigl(\phi(p^*)\bigr)\right)_{i,j},
\]
is invertible. This condition is independent of the choice of local coordinates. See \cite{jmlee}.
\begin{definition}[Morse function \cite{MT}]\label{def:morse}
A smooth function \(f:M\to\mathbb R\) is called a Morse function if
all of its critical points are non-degenerate.
\end{definition}

Denote by
\[
\Sigma:
=
\{p\in M:(df)_p=0\}
\]
the set of critical points of \(f\). If \(f\) is Morse and \(M\) is compact, then every critical point of \(f\) is isolated; consequently, \(\Sigma\) is a finite set.

\subsection{The geodesic normal coordinates}
We fix a critical point \(p^* \in \Sigma\).

\subsubsection{Eigenvalues of the Riemannian Hessian}

The eigenvalues $\kappa_1(p^*),\dots,\kappa_N(p^*)$ of the Riemannian Hessian operator $H_{p^*}(f)$: $T_{p^*} M\to T_{p^*} M$ associated with the metric $g$ are defined by the relation
\[
\operatorname{Hess}_{p^*} (f)(v, w) =g(H_{p^*}(f)(v),w), 
\quad \forall  v,w\in T_{p^*} M.
\]
That is, for each $i = 1, \ldots, N$, there exists a nonzero vector $e_i \in T_p M$ such that
\[H_{p^*}(f)(e_i)=\kappa_i({p^*})e_i.
\]
This definition is intrinsic and then we have
\begin{align}\label{eij}
    \operatorname{Hess}_{p^*}(f) (e_i, w) = \kappa_i(p^*) g(e_i, w) \quad\text{for all } w\in T_{p^*}M.
\end{align}
Since $\operatorname{Hess}_{p^*}(f)$ is symmetric, the corresponding
linear operator is self-adjoint with respect to the metric $g$.
Therefore, the eigenvectors
$e_1,\ldots,e_N$, can be chosen to form a $g$-orthonormal
basis of $T_{p^*}M$, i.e.,
\[
g(e_i,e_j)=\delta_{ij},
\qquad i,j=1,\ldots,N.
\]
Taking $w=e_j$ in \eqref{eij}, we obtain the diagonalization of the Hessian at $p^*$:
\begin{equation}\label{eq:hess-eigenbasis}
\operatorname{Hess}_{p^*}f(e_i,e_j)
=
\kappa_i g(e_i,e_j)
=
\kappa_i(p^*)\delta_{ij}.
\end{equation}

\subsubsection{The geodesic normal coordinates}

Let \(\varphi: U_R = B_g(p^*,R) \to B(0,R) \subset \mathbb R^N\) be a geodesic normal coordinate chart $(U_R,\varphi)$ around \(p^*\), with \(R>0\) sufficiently small. By the standard construction of geodesic normal coordinates (see, e.g., \cite{RGA}), this chart can be chosen such that its coordinate basis at \(p^*\) coincides with the \(g\)-orthonormal eigenbasis \(\{e_1,\ldots,e_N\}\) determined above. That is, 
\[
\varphi(p)=x=(x_1,...,x_N) \text{ for } p\in U_R, \quad \varphi(p^*)=0,  \quad 
\left.\frac{\partial}{\partial x_i}\right|_{p^*} = e_i, \quad i=1,\ldots,N.
\]
In these coordinates, the metric satisfies the well-known properties
\[
g_{ij}(0)=\delta_{ij}, \qquad
\Gamma_{ij}^k(0)=0, \qquad
\frac{\partial g_{ij}}{\partial x_k}(0)=0,
\]
where \(\Gamma_{ij}^k\) are the Christoffel symbols of the Levi-Civita connection of \(g\). Consequently, as \(x\to0\),
\begin{equation}\label{tay}
g_{ij}(x)
=
\delta_{ij}+O(|x|^2),
\quad
g^{ij}(x)
=
\delta^{ij}+O(|x|^2),
\quad
\sqrt{\det(g_{ij}(x))}
=
1+O(|x|^2).
\end{equation}

\subsubsection{Local expansion of the Morse function $f$ near $p^*$}

Let $(U_R,\varphi)$ be the geodesic normal coordinate chart constructed in the previous subsection, and write $\hat f = f \circ \varphi^{-1}$ for the coordinate representation of $f$. Since $p^*\in\Sigma$ is a critical point, we have
\[
\frac{\partial \hat f}{\partial x_i}(0)=0,\qquad i=1,\ldots,N.
\]
Moreover, by the choice of the coordinate basis $\{e_i\}$ and the definition of the Hessian eigenvalues, equation \eqref{eq:hess-eigenbasis} yields
\[
\frac{\partial^2 \hat f}{\partial x_i \partial x_j}(0)
=
\operatorname{Hess}_{p^*}(f)(e_i,e_j)
=
\kappa_i(p^*)\,\delta_{ij}.
\]
Thus the Hessian matrix at the origin is diagonal:
\[
\left(
\frac{\partial^2 \hat f}{\partial x_i \partial x_j}(0)
\right)_{1\le i,j\le N}
=
\operatorname{diag}(\kappa_1(p^*),\ldots,\kappa_N(p^*)).
\]
Applying Taylor's theorem around $x=0$, and noting that there are no linear terms, we obtain the local expansion
\begin{align}
\hat f(x)
&=
\hat f(0)
+
\sum_{i=1}^N
\frac{\partial \hat f}{\partial x_i}(0)x_i
+
\frac12
\sum_{i,j=1}^N
\frac{\partial^2 \hat f}{\partial x_i \partial x_j}(0)x_i x_j
+
O(|x|^3)
\nonumber\\
&=
f(p^*)
+
\frac12\sum_{i=1}^N \kappa_i(p^*)\,(x_i)^2
+
O(|x|^3),
\qquad x\to0.
\label{localexpansionf}
\end{align}

Consequently, in the subsequent analysis, whenever we perform local computations around a critical point, we shall always adopt such geodesic normal coordinates with the coordinate axes aligned with the $g$-orthonormal eigenbasis of the Riemannian Hessian at that point.

\subsection{Variational characterization of $\lambda(D)$.}
The equation 
\begin{equation*}
   -D\Delta_{M}u-a\langle  \nabla_M f, \nabla_M u\rangle_g +c u=\lambda(D)u,
\end{equation*}
can be rewritten in the following divergence form
\begin{equation*}
   -D\,\dive_g\left(e^{\frac{a}{D} f}\nabla_{M}u\right)+e^{\frac{a}{D} f} c u=\lambda(D)e^{\frac{a}{D} f} u.   
\end{equation*}
Thus, we have the following variational characterization.

\textbf{Variational characterization 1.} The principal eigenvalue $\lambda(D)$ is characterized by
\begin{equation}\label{var}
  \lambda(D)=\min_{\substack{u\in H^1(M,\V)\\ u\not\equiv 0}}
  \frac{\displaystyle\int_M e^{\frac{a}{D} f}\left(D|\nabla_{M}u|^2 + c |u|^2\right)\V}
  {\displaystyle\int_M e^{\frac{a}{D} f} u^2 \V},
\end{equation}
where $\V$ denotes the Riemannian volume element induced by the metric. The derivation of this variational formula is standard (see, e.g., \cite{zhoudetang,XZ26}). It is immediate that 
\[
c_*\le \lambda(D)\le c^*, 
\quad\text{with}\quad
c_*=\min_{p\in M} c(p),\quad c^*=\max_{p\in M} c(p).
\]
Without loss of generality, we assume throughout that \(c_*>0\).

Let \(u(D,\cdot)\) be the unique positive principal eigenfunction corresponding to $\lambda(D)$ in \eqref{E}, normalized by
\[
\int_{M} e^{ \frac{a}{D}f}u^2\V =1.
\]
Setting $w=e^{\frac{a f}{2 D} }u$, we obtain $\|w(D, \cdot)\|_{L^2(M,\V)}=1$ and $w$ satisfies
\begin{equation}\label{Ew}
   -D\Delta_{M}w+\left(\frac{a^2}{4 D}  |\nabla_M f|^2+ \frac{a}{2}\Delta_M f +c\right) w=\lambda(D)w.   
\end{equation}

Since for any $D>0$, $\int_M w^2(D,\cdot)\V=1$ and $M$ is compact, by the Prokhorov theorem (see \cite{yan}), there exist a subsequence $\{D_j\}_{j=1}^{+\infty}$ and a probability measure $\mu$ on $M$ such that for every $\zeta\in C(M)$,
\begin{equation}\label{weak}
\lim_{j\to+\infty}\int_M w^2(D_j,\cdot)\,\zeta\,\V=\int_M \zeta\mu(\mathrm{d} \mathrm{V_g})=\int_M \zeta\,\mathrm d\mu,
\end{equation}
where $C(M)$ denotes the space of continuous functions on $M$. By investigating the support of $\mu$, we can determine the asymptotics of $\lambda(D)$. 

From \eqref{Ew}, we also obtain the following equivalent variational characterization.

\textbf{Variational characterization 2.} The principal eigenvalue $\lambda(D)$ satisfies
\begin{equation}\label{var2}
  \lambda(D)=\min_{\substack{w\in H^1(M,\V)\\ \int_M w^2\V=1}}
  \int_M
 \left(D\left|\nabla_{M} w-\frac{a}{2 D} w\nabla_{M} f\right|^2 + c |w|^2 \right)\V.
\end{equation}

Define the functional
\begin{equation}\label{Eu}
  E(u):=\int_M
 D\left|\nabla_{M} u-\frac{a}{2D} u\nabla_{M} f\right|^2 \V
 =\int_M
 D\left|\nabla_{M} \ln{u}-\frac{a}{2D} \nabla_{M} f\right|^2 u^2 \V,
\end{equation}
where the second equality holds for positive \(u\). The following lemma provides a lower bound for \(E(u)\).
\begin{lemma}\label{lem:2.1}
For every \(u\in H^1(M,\V)\), we have 
\[
E(u)\ge \frac{a^2}{4D}  \int_{M} \left( | \nabla_M f|^2 +
     \frac{2D}{a} \Delta_M f \right) u^2 \V.
\]
\end{lemma}
\begin{proof}
Expanding the square and integrating by parts, we obtain
\[ \begin{aligned}
     E (u) & = \int_{M}\left[ D | \nabla_M u|^2 
     -  a u \langle \nabla_M u,\nabla_M f\rangle
 + \frac{a^2}{4D} u^2  | \nabla_M f|^2\right]\V\\
     & = \int_{M}\left[ D | \nabla_M u|^2 +
     \frac{a}{2} u^2 \Delta_M f +
     \frac{a^2}{4D} u^2  | \nabla_M f|^2\right]\V\\
     & \ge \int_{M}\left[
     \frac{a^2}{4D} u^2  | \nabla_M f|^2 +
     \frac{a}{2} u^2 \Delta_M f
     \right]\V \\
     & = \frac{a^2}{4D}  \int_{M} \left( | \nabla_M f|^2 +
     \frac{2D}{a} \Delta_M f \right) u^2 \V.
   \end{aligned} \]
This proves the lemma.
\end{proof}
Lemma \ref{lem:2.1} indicates that, as $D \rightarrow 0$, the $L^2$-mass of the principal eigenfunction $w$ concentrates near the critical points of $f$. It is therefore natural to study the local behavior of $w$ near those points.

Let $\zeta$ be a smooth function. Multiplying the differential equation \eqref{Ew}
by $\zeta^2 w$ and integrating over $M$, we obtain
\[ \begin{aligned}
     &\int_{M} \lambda (D) \zeta^2 w^2 \V\\
     =& \int_{M}\left[- D \zeta^2 w
     \Delta_M w + \frac{a^2}{4D}  | \nabla_M f|^2 \zeta^2 w^2 +\frac{a}{2} \zeta^2
     w^2 \Delta_M f + c \zeta^2 w^2\right]\!\mathrm{d} \mathrm{V_g}\\
      =&\int_{M}\left[D \nabla_M w \cdot \nabla_M (\zeta^2 w) +
     \frac{a^2}{4D}  | \nabla_M f|^2 \zeta^2 w^2 + \frac{a}{2} \zeta^2 w^2 \Delta_M
     f + c \zeta^2 w^2\right]\!\mathrm{d} \mathrm{V_g} .
   \end{aligned} \]

Set $W = \zeta w$. Then the above identity can be rearranged as
\[ \begin{aligned}
     \int_{M} \lambda (D) \zeta^2 w^2 \V =& \int_{M} D \nabla_M w \cdot
     (\zeta \nabla_M (\zeta w) + \zeta w \nabla_M \zeta)\\
     & + \frac{a^2}{4D}  | \nabla_M f|^2 \zeta^2 w^2 -  a \zeta w
     \nabla_M (\zeta w) \cdot \nabla_M f + c \zeta^2 w^2\mathrm{d} \mathrm{V_g}\\
     =&  \int_{M} D \left| \nabla_M (\zeta w) - \frac{a}{2D} \zeta w
     \nabla_M f \right|^2\\
     & + c \zeta^2 w^2 - Dw \nabla_M \zeta \cdot \nabla_M (\zeta w) + D \zeta w
     \nabla_M w \cdot \nabla_M \zeta\mathrm{d} \mathrm{V_g}\\
      =& \int_{M} D \left| \nabla_M W - \frac{a}{2D} W \nabla_M f
     \right|^2 + c W^2 - D \int_{M} \frac{W^2}{\zeta^2}  | \nabla_M \zeta
     |^2 \mathrm{d} \mathrm{V_g}.
   \end{aligned} \]
Thus we have the following useful identity.
\begin{lemma}\label{lem:2.2}
Let \((\lambda(D), w)\) be a solution of \eqref{Ew} and \(\zeta\) be smooth. Define \(W := \zeta w\). Then
\begin{equation}
\int_M \left( D \left| \nabla_M W - \frac{a}{2D} W \nabla_M f \right|^2 + c W^2 \right) \V =\! \lambda(D) \int_M W^2 \V + D \int_M \frac{W^2}{\zeta^2} |\nabla_M \zeta|^2 \V.
\end{equation}
\end{lemma}


\section{Proof of theorem}
Recall that \(\Sigma\) denotes the set of critical points of \(f\), i.e.,
\[
\Sigma := \{ p \in M : \nabla_M f(p) = 0 \}.
\]
Since \(M\) is closed and \(f\) is a Morse function, \(\Sigma\) is nonempty and consists of isolated points. For \(p\in\Sigma\), we adopt the geodesic normal coordinates \((U_R,\varphi)\) centered at \(p\), as constructed in the previous subsection, and write \(\hat f = f\circ\varphi^{-1}\) for the coordinate representation of \(f\).
\subsection{Upper bound estimate}
We first prove the upper bound for the limit.
\begin{proposition}\label{prop:upper} 
For the principal eigenvalue $\lambda(D)$, we have the following estimate:
\[
\limsup_{D \to 0} \lambda(D) \leq \min_{p \in \Sigma} \left\{ c(p) + \frac{a}{2} \sum_{i=1}^N \left( |\kappa_i(p)| + \kappa_i(p) \right) \right\}.
\]
\end{proposition}
\begin{proof}
For any fixed \(p^*\in\Sigma\), we shall construct a test function supported near \(p^*\). Let \((U_{p^*},\varphi)\) be the geodesic normal coordinates centered at \({p^*}\) constructed in the previous subsection with $U_{p^*}\subset U_R$. In these coordinates we have \(g_{ij}(0)=\delta_{ij}\) and \(\Gamma_{ij}^k(0)=0\), and consequently the metric expansions \eqref{tay} hold.
Thus there exists a constant $C_{p^*} > 0$, depending only on $g$ and $p^*$, such that $|g^{ij}(x) - \delta^{ij}| \le C_{p^*}|x|^2$ for $x$ sufficiently close to $0$. Therefore, for any given $\varepsilon> 0$, we may choose a sufficiently small radius $\tau = \sqrt{\varepsilon / C_{p^*}} > 0$ such that for all $x \in B(0, \tau)$, the following strict bound holds:
\[
|g^{ij}(x) - g^{ij}(0)| < \varepsilon.
\]
We take $U_p$ small enough that its image $\varphi(U_p)$ is contained in the ball $B(0, \tau)$. In this neighborhood, by \eqref{localexpansionf}, the function $\hat{f}(\cdot)=f\circ \varphi^{-1}(\cdot)$ has the following expansion near the origin:
\[
\hat{f}(x) = \hat{f}(0) + \frac{1}{2} \sum_{i=1}^N \kappa_i({p^*}) x_i^2 + O(|x|^3),
\]
where $\kappa_1({p^*}), \ldots, \kappa_N({p^*})$ are the eigenvalues of Hessian operator $H_{p^*}(f)$. Consequently, its first derivatives satisfy:
\begin{equation}\label{tayf}
\partial_{x_i}\hat{f}(x) = \kappa_i({p^*}) x_i + O(|x|^2), \quad \text{for } |x| \leq \tau.
\end{equation}
For simplicity, we write $\kappa_i$ for $\kappa_i({p^*})$ in what follows. For the sake of convenience, here and throughout, let $C$ denote a generic positive constant (depending only on $a$, $g$, ${p^*}$, $N$ and the $\kappa_i$'s) whose value may change from line to line.

Fix a small positive constant \(0<\delta < 1\) and set \( L :=D^{-\frac{1}{8}}>0\). Define on \(\R^N\) for $i=1,2,...,N$ that 
\[
p_i(x_i) = \left[ e^{-\frac{1}{2}(|\kappa_i| + \delta) x_i^2} - e^{-\frac{1}{2}(|\kappa_i| + \delta) L^2} \right]^+,\quad p(x) = \prod_{i=1}^N p_i(x_i),\quad q_i(x_i)=e^{-\frac{1}{2}(|\kappa_i| + \delta) x_i^2},
\]
and
\begin{equation} \label{Ci}
        C_i(x_i)=q_i(x_i)-p_i(x_i)=
\begin{cases}
e^{-\frac12(|\kappa_i|+\delta)L^2}, & |x_i|<L,\\[1ex]
e^{-\frac12(|\kappa_i|+\delta)x_i^2}, & |x_i|\ge L.
\end{cases}
\end{equation}
Let \(\rho = \sqrt{2D/a}\), where $a>0$ is fixed and set
\[
\zeta(x) = \frac{1}{\rho^{N/2}} p\left( \frac{x}{\rho} \right).
\]
Choose \(0<r_0<\tau\) sufficiently small so that \(B(0, r_0) \subset \varphi(U_{p^*})\) and \eqref{tay} are valid. For \(D\) sufficiently small such that \( L < \frac{r_0}{\rho \sqrt{N}}\), we have
\begin{align}
 C(\rho, \delta, L) &= \int_{B(0, r_0)} \zeta^2(x)\sqrt{\mathrm{det}g_{ij}(x)}\mathrm{d} x\nonumber\\
&= \int_{B(0, r_0)} \zeta^2(x)\mathrm{d} x+\int_{B(0, r_0)} \zeta^2(x)O(|x|^2)\mathrm{d} x\nonumber\\
 &\overset{x=\rho y}{=}\int_{\mathbb{R}^N} p^2(y) \mathrm{d}y+\rho^2\int_{\mathbb{R}^N} p^2(y) O(|y|^2)\mathrm{d}y\nonumber\\
 &=:C_1(\delta,L)+\rho^2C_2(\delta,L),
\end{align}
where
$$C_1(\delta,L)=\int_{\mathbb{R}^N} p^2(y) \mathrm{d}y,\quad C_2(\delta,L)=\int_{\mathbb{R}^N} p^2(y) O(|y|^2)\mathrm{d}y.$$
Note that 
$$C_1(\delta,L)\to {\prod_{i=1}^N\sqrt{\frac{\pi}{|\kappa_i|+\delta}}}:=C(\delta), \quad \text{as } L\to\infty.$$ 
Moreover, the second term $\rho^2C_2(\delta,L)$ satisfies the two-sided estimate 
$$-C\rho^2L^2C(\delta)\le \rho^2C_2(\delta,L)=\rho^2\int_{\mathbb{R}^N} p^2(y) O(|y|^2)\mathrm{d}y\le C\rho^2L^2C_1(\delta,L)\le C\rho^2L^2C(\delta),$$
hence
$$
C_1(\delta, L)-C D^{3/4}C(\delta)\leq C(\rho,\delta,L)\leq C_1(\delta, L)+C D^{3/4}C(\delta).
$$
As $D\to0$ (equivalently, $L\to \infty$ and $\rho\to0$), we have $C(\rho,\delta,L)\to C(\delta)$. Therefore, for sufficiently small $D$, the following holds:
$$0<\frac{1}{2}C(\delta)<C(\rho,\delta,L)<\frac{3}{2}C(\delta).$$


Denote $\hat{w}(x)=\frac{\zeta(x)}{\sqrt{C(\rho,\delta,L)}}$. Then $$\int_M w^2 \V=\int_{\varphi(U_p)}\hat{w}^2(x)\sqrt{\mathrm{det}g_{ij}(x)}\mathrm{d} x=1.$$  For brevity, we write $\sqrt{g(x)}:=\sqrt{\mathrm{det}g_{ij}(x)}$ in what follows. Using the variational formula \eqref{var2}, we obtain
\begin{align}\label{I:upe}
  \lambda(D)\leq & \int_M  \left(D|\nabla_{M} w-\frac{a}{2 D} w\nabla_{M} f|^2 +c |w|^2 \right)\mathrm{d} \mathrm{V_g}\nonumber\\
 \leq &\int_{\varphi(U_p)}  D g^{ij}(x)(\frac{\partial \hat{w}}{\partial x_i} -\frac{a}{2 D} \hat{w}\frac{\partial \hat{f}}{\partial x_i} )(\frac{\partial \hat{w}}{\partial x_j} -\frac{a}{2 D} \hat{w}\frac{\partial \hat{f}}{\partial x_j} )
 \sqrt{g(x)}\mathrm{d} x\nonumber\\
 &+\int_{\varphi(U_p)} \hat{c}(x) \hat{w}^2(x)
 \sqrt{g(x)}\mathrm{d} x
 \nonumber\\
 = &   \int_{\varphi(U_p)} D g^{i j}
  (0) \sqrt{g(0)}  \left( \frac{\partial \hat{w}}{\partial x_i} -
  \frac{a}{2D} \hat{w} \frac{\partial \hat{f}}{\partial x_i} \right)  \left(
  \frac{\partial \hat{w}}{\partial x_j} - \frac{a}{2D} \hat{w} \frac{\partial
  \hat{f}}{\partial x_j} \right)  \mathrm{d} x  \nonumber\\
  & + \int_{\varphi(U_p)} D \sqrt{g(0)}  (g^{i j} (x) - g^{i j} (0)) 
  \left( \frac{\partial \hat{w}}{\partial x_i} - \frac{a}{2D} \hat{w} \frac{\partial
  \hat{f}}{\partial x_i} \right)  \left( \frac{\partial \hat{w}}{\partial x_j} -
  \frac{a}{2D} \hat{w} \frac{\partial \hat{f}}{\partial x_j} \right)  \mathrm{d} x
  \nonumber\\
  & + \int_{\varphi(U_p)} D \left( \sqrt{g(x)}  - \sqrt{g(0)}  \right) g^{i j} (x) \left( \frac{\partial \hat{w}}{\partial x_i} -
  \frac{a}{2D} \hat{w} \frac{\partial \hat{f}}{\partial x_i} \right)  \left(
  \frac{\partial \hat{w}}{\partial x_j} - \frac{a}{2D} \hat{w} \frac{\partial
  \hat{f}}{\partial x_j} \right)  \mathrm{d} x \nonumber\\
   & +\int_{\varphi(U_p)} \hat{c}(x) \hat{w}^2(x)
 \left(\sqrt{g(x)} -\sqrt{g(0)} \right)\mathrm{d} x
 +\int_{\varphi(U_p)} \hat{c}(x) \hat{w}^2(x)
 \sqrt{g(0)} \mathrm{d} x
\nonumber\\
  := & L_1 + L_2 + L_3+L_4+L_5.
\end{align}

Next, we estimate the upper bound of each term, $L_1,L_2,L_3,L_4,L_5$, respectively. 

\textbf{Estimates for $L_2,L_3$:} Let $A(x) = (a^{ij}(x))_{N \times N}$ with $a^{ij}(x) = g^{ij}(x) - g^{ij}(0)$. Recall that $g^{ij}(x) = \delta^{ij} + O(|x|^2)$ on $\varphi(U_{p^*})$, and there exists a positive constant $C_{p^*}$ such that $\forall i,j=1...n.$, $
|a^{ij}(x)| \le C_{p^*}|x|^2 \le C_{p^*}\tau^2$ for all $x \in \varphi(U_{p^*}) \subset B(0, \tau)$. Consequently, for any $\xi = (\xi_1, \ldots, \xi_N) \in \mathbb{R}^N$, 
\begin{align*}
\left| \sum_{i,j=1}^N a^{ij}(x) \xi_i \xi_j \right| \le C_{p^*}\tau^2 \sum_{i,j=1}^N |\xi_i| |\xi_j| \le \frac{C_{p^*}\tau^2}{2}  \sum_{i,j=1}^N (\xi_i^2 + \xi_j^2)  = C_{p^*} N \tau^2 |\xi|^2.
\end{align*}
Recall the definition of $L_1$ and the fact $g^{ij}(0)=\delta^{ij}$. Applying this with $\xi_i = \frac{\partial \hat{w}}{\partial x_i} - \frac{a}{2D} \hat{w} \frac{\partial \hat{f}}{\partial x_i}$, we get
\begin{align*}
|L_2| 
\le C_p N \tau^2 \int_{\varphi(U_p)} D \sum_{i=1}^N \left( \frac{\partial \hat{w}}{\partial x_i} - \frac{a}{2D} \hat{w} \frac{\partial \hat{f}}{\partial x_i} \right)^2 \mathrm{d}x = C_{p^*} N \tau^2 L_1.
\end{align*}
Since $\tau=\sqrt{\varepsilon / C_p}$,  we obtain
\[
|L_2| \le C\varepsilon L_1.
\]
 Similarly, $|L_3|\leqslant
C\varepsilon L_1$.

\textbf{Estimates for $L_4$:}  Recall that $c_*>0$, which implies that $\hat{c}>0$. Using the expansion $\sqrt{ g(x)} = \sqrt{g(0)} + O(|x|^2)$ and the fact that $\tau^2=\varepsilon/C_p$, we have 
\begin{equation*}
    |L_4|\le\int_{\varphi(U_p)} \hat{c}(x) \hat{w}^2(x)
 \left|\left(\sqrt{g(x)}-\sqrt{g(0)}\right)\right|\mathrm{d} x\le C\varepsilon L_5.
\end{equation*}

Combining this with the estimates for $L_2$ and $L_3$, we have controlled all error terms arising from the metric expansion by their corresponding principal terms. Substituting these bounds into \eqref{I:upe} yields
\begin{align}\label{UBL}
\lambda(D) &\le L_1 + L_2 + L_3 + L_4 + L_5 \le (1 + C\varepsilon) (L_1 + L_5).
\end{align}
Therefore, to establish the upper bound for $\lambda(D)$, we reduce to estimating the principal terms $L_1$ and $L_5$.

\textbf{Estimates for $L_1+L_5$:} 
By the definition of $\hat{w}(x)$, we calculate that
\begin{align*}
L_1+L_5=&
\int_{\varphi(U_p)} D  \left|\nabla \hat{w} -
  \frac{a}{2D} \hat{w} \nabla \hat{f} \right|^2\mathrm{d} x
  +\int_{\varphi(U_p)} \hat{c}(x) \hat{w}^2(x) \mathrm{d} x\\
  =&
\int_{B(0,\rho L)} D  \left|\nabla \hat{w} -
  \frac{a}{2D} \hat{w} \nabla \hat{f} \right|^2\mathrm{d} x
  +\int_{B(0,\rho L)} \hat{c}(x) \hat{w}^2(x) \mathrm{d} x
  \\=&
  \int_{\mathbb{R}^N} \left[ D  \left|\nabla \ln(\hat{w}(x)) -
  \frac{a}{2D} \nabla \hat{f}(x) \right|^2+\hat{c}(x) \right] \hat{w}^2(x) \mathrm{d}x
  \\=&
 \frac{1}{C(\rho,\delta,L)} \int_{\mathbb{R}^N} \left[ D  \left|\nabla \ln(\hat{w}(x)) -
  \frac{a}{2D} \nabla \hat{f}(x) \right|^2+\hat{c}(x) \right] \zeta^2(x) \mathrm{d}x
 \end{align*}
Thanks to \eqref{tayf} and the fact $\nabla \ln\hat{w}(x)=\nabla \ln \zeta(x)$, by the definition of $\zeta(x)$ and \(\rho = \sqrt{2D/a}\), we have
\begin{align*}
\!\!&L_1+L_5\\
=\!&
 \int_{\mathbb{R}^N} \left[D\sum_{i=1}^N\left(\frac{(|\kappa_i|+\delta)x_ie^{-\frac{(|\kappa_i|+\delta)x_i^2}{2\rho^2}}}{\rho^2 p_i(\frac{x_i}{\rho})}
 +\frac{a}{2D}\kappa_ix_i
 +\frac{a}{2D}O(|x|^2)\right)^2+\hat{c}(x) \right]  
 \frac{\zeta^2(x)}{C(\rho,\delta,L)}\mathrm{d}x
  \\= \!&
\int_{\mathbb{R}^N} \left[ D\sum_{i=1}^N\left(\frac{(|\kappa_i|+\delta)y_ie^{-\frac{(|\kappa_i|+\delta)y_i^2}{2}}}
 {\rho p_i(y_i)}+\frac{a \rho}{2D}\kappa_iy_i+\frac{a\rho^2}{2D} O(|y|^2)\right)^2
 +\hat{c}(\rho y) \right]
  \frac{p^2(y)}{C(\rho,\delta,L)} \mathrm{d}y
  \\= \!&
 \int_{\mathbb{R}^N} 
 \left[\frac{a}{2}\sum_{i=1}^N\left(\frac{(|\kappa_i|+\delta)y_i
 q_i(y_i)}
 { p_i(y_i)}+ \kappa_iy_i+ \rho O(|y|^2)\right)^2
 +\hat{c}(\rho y) \right] 
  \frac{p^2(y)}{C(\rho,\delta,L)} \mathrm{d}y
   \\= \!&
 \int_{\mathbb{R}^N} 
 \left[\frac{a}{2}\sum_{i=1}^N\left(\frac{(|\kappa_i|+\delta)y_i C_i(y_i)}
 { p_i(y_i)}+(|\kappa_i|+\delta+\kappa_i)y_i+ \rho O(|y|^2)\right)^2
 \!\!+\hat{c}(\rho y) \right] 
  \frac{p^2(y)}{C(\rho,\delta,L)} \mathrm{d}y
 \\ =\! &
 \int_{\mathbb{R}^N}\Bigg\{
\frac{a}{2}\sum_{i=1}^N
\Bigg[
\left( |\kappa_i| + \kappa_i + \delta \right)^2 y_i^2
+
\frac{C_i^2 y_i^2}{p_i(y_i)^2} (|\kappa_i| + \delta)^2
+
\frac{2 C_i (|\kappa_i| + \kappa_i + \delta)(|\kappa_i|+\delta)y_i^2}{p_i(y_i)}
\\
&+\! \rho^2 O(|y|^4)+(|\kappa_i|\! +\! \kappa_i\! +\! \delta)y_i\rho  O(|y|^2)\!
\!+\!\!\frac{(|\kappa_i|\!+\!\delta)y_iC_i(y_i)}
 { p_i(y_i)} \rho  O(|y|^2)\!\Bigg]\!\!
+\! \hat{c}(\rho y)\!
\Bigg\}\!\frac{p^2(y)}{C(\rho,\delta,L)}\mathrm{d}y
\\:= &J_1+J_2+J_3+J_4+J_5,
\end{align*}
where
\begin{align*}
J_1\!&=\!\int_{\mathbb{R}^N}
\frac{a}{2}\sum_{i=1}^N
\left( |\kappa_i| + \kappa_i + \delta \right)^2 y_i^2
\frac{p^2(y)}{C(\rho,\delta,L)}\mathrm{d}y\\
J_2\!&=\!\int_{\mathbb{R}^N}
\frac{a}{2}\sum_{i=1}^N
\frac{C_i^2 y_i^2}{p_i(y_i)^2} (|\kappa_i| + \delta)^2
\frac{p^2(y)}{C(\rho,\delta,L)}\mathrm{d}y\\
J_3\!&=\!\int_{\mathbb{R}^N}
a\sum_{i=1}^N
\frac{ C_i (|\kappa_i| + \kappa_i + \delta)(|\kappa_i|+\delta)y_i^2}{p_i(y_i)}
\frac{p^2(y)}{C(\rho,\delta,L)}\mathrm{d}y\\
J_4\!&=\!\int_{\mathbb{R}^N}\!
\sum_{i=1}^N\! \Bigg[\!
\rho^2 O(|y|^4)\!+\!(|\kappa_i|\! +\! \kappa_i\! +\! \delta)y_i\rho  O(|y|^2)
\!+\!\frac{(|\kappa_i|+\delta)y_iC_i(y_i)}
 { p_i(y_i)} \rho  O(|y|^2)\!\Bigg]\!
\frac{p^2(y)}{C(\rho,\delta,L)}\mathrm{d}y\\
J_5\!&=\!\int_{\mathbb{R}^N} \hat{c}(\rho y)
\frac{p^2(y)}{C(\rho,\delta,L)}\mathrm{d}y
\end{align*}

\textbf{Estimate of $J_1$.}
From the definition of $p(y)$ and the bounds on $C(\rho,\delta,L)$, for sufficiently small $D$, we have  
\begin{align*}
J_1=&
\frac{a}{2C(\rho,\delta,L)} \sum_{i=1}^N \int_{\mathbb{R}^N} 
(|\kappa_i|+\kappa_i+\delta)^2 y_i^2 p^2(y) \mathrm{d}y\nonumber\\
\leqslant & 
\frac{a}{2C(\rho,\delta,L)} \sum_{i=1}^N \int_{\mathbb{R}^N} 
  e^{- \sum_{j = 1}^N (|k_j | + \delta) y_j^2}  (|k_i | +
  \delta + k_i)^2 y_i^2 \textrm{ d} y \nonumber\\
  ={}&
\frac{a}{2C(\rho,\delta,L)}
\sum_{i=1}^N
(|\kappa_i|+\delta+\kappa_i)^2
\left(
\int_{\mathbb R}
y_i^2e^{-(|\kappa_i|+\delta)y_i^2}\,\mathrm{d}y_i
\right)
\prod_{\substack{ j\neq i}}^N
\left(
\int_{\mathbb R}
e^{-(|\kappa_j|+\delta)y_j^2}\,\mathrm{d}y_j
\right)
\nonumber\\
  = & 
\frac{a}{2C(\rho,\delta,L)} \left( \sum_{i= 1}^N \frac{(|\kappa_i | +
  \delta + \kappa_i)^2}{2 (|k_i | + \delta)} \right) \left( \overset{N}{\underset{j
  = 1}{\prod}} \sqrt{\frac{\pi}{|\kappa_j| + \delta}}  \right) \nonumber\\
  =&
  \frac{a}{2}\frac{C(\delta)}{C(\rho,\delta,L)} \left( \sum_{i = 1}^N \frac{(|\kappa_i | +
  \delta + \kappa_i)^2}{2 (|\kappa_i | + \delta)} \right) 
\end{align*}
As $D\to 0$, since $C(\rho,\delta,L)\to C(\delta)$, we have
\begin{equation}\label{J1}
J_1\to \frac{a}{2}\sum_{i = 1}^N \frac{(|\kappa_i | +
  \delta + \kappa_i)^2}{2 (|\kappa_i| + \delta)}.    
\end{equation}

\textbf{Estimate of $J_2$.} 
Thanks to \eqref{Ci}, we have the following inequalities
\begin{align*}
&\int_{\mathbb{R}}
{(|\kappa_i| + \delta)^2 C_i^2  y_i^2}\mathrm{d}y_i\\
=&\int_{|y_i|<L}
{(|\kappa_i| + \delta)^2   e^{-(|\kappa_i|+\delta)L^2}y_i^2}\mathrm{d}y_i+\int_{|y_i|\ge L}
{(|\kappa_i| + \delta)^2 e^{-(|\kappa_i|+\delta)y_i^2}  y_i^2}\mathrm{d}y_i    \\
\le& (|\kappa_i| + \delta)^2   L^3e^{-(|\kappa_i|+\delta)L^2}+\int_{|y_i|\ge L}
{(|\kappa_i| + \delta)^2 e^{-(|\kappa_i|+\delta)y_i^2}  y_i^2}\mathrm{d}y_i  \\\le& C_{\kappa}   L^3e^{-(|\kappa_i|+\delta)L^2}+\int_{|y_i|\ge L}
{(|\kappa_i| + \delta)^2 e^{-(|\kappa_i|+\delta)y_i^2}  y_i^2}\mathrm{d}y_i. 
\end{align*}
where $C_{\kappa}>0$ is a constant depending on the $\kappa_i$'s.

We now turn to $J_2$. Since $p_j^2(y_j) \le e^{-(|\kappa_j|+\delta)y_j^2}$ for all $j$, it follows from the above estimate that
\begin{align*}
J_2&=\int_{\mathbb{R}^N}
\frac{a}{2}\sum_{i=1}^N
\frac{C_i^2 y_i^2}{p_i(y_i)^2} (|\kappa_i| + \delta)^2
\frac{p^2(y)}{C(\rho,\delta,L)}\mathrm{d}y\\
&=\frac{a}{2C(\rho,\delta,L)}\sum_{i=1}^N\left[
\int_{\mathbb{R}}
{(|\kappa_i| + \delta)^2 C_i^2  y_i^2}\mathrm{d}y_i
\left(\prod_{j\neq i}\int_{\mathbb{R}} {p_j^2(y_j)}\mathrm{d}y_j \right)\right]\\
&\leq
\frac{C }{C(\rho,\delta,L)}
\sum_{i=1}^N   \left[
 \left(\prod_{j\neq i}\int_{\mathbb{R}} e^{-(|\kappa_j|+\delta)y_j^2}\mathrm{d}y_j\right)\int_{\mathbb{R}}
{(|\kappa_i| + \delta)^2 C_i^2  y_i^2}\mathrm{d}y_i\right]
 \nonumber\\
&\le \frac{C }{C(\rho,\delta,L)}
\sum_{i=1}^N   
 \prod_{j \neq i} \sqrt{\frac{\pi}{|\kappa_j|+\delta}}\left(    L^3e^{-(|\kappa_i|+\delta)L^2}+\int_{|y_i|\ge L}
{(|\kappa_i| + \delta)^2 e^{-(|\kappa_i|+\delta)y_i^2}  y_i^2}\mathrm{d}y_i     \right)
 \nonumber\\
 &\le \frac{C }{C(\rho,\delta,L)}\sum_{i=1}^N   
 \left(     L^3e^{-(|\kappa_i|+\delta)L^2}+\int_{|y_i|\ge L}
{ e^{-(|\kappa_i|+\delta)y_i^2}  y_i^2}\mathrm{d}y_i     \right),
\end{align*}
where $C>0$ is a constant depending on the $\kappa_i$'s and $N$, and we have absorbed the factor $(|\kappa_i|+\delta)^2$ into $C$ using the fact that $0<\delta<1$ is bounded.

Since 
\[
\int_{|y_i|\ge L} e^{-(|\kappa_i|+\delta)y_i^2} y_i^2\,\mathrm{d}y_i \longrightarrow 0
\qquad \text{as } L\to\infty.
\]
Consequently, as \(D\to 0\) (and hence \(L\to 0\)), the bound 
$
\frac12 C(\delta) < C(\rho,\delta,L) < \frac32 C(\delta)
$
from the preceding estimate ensures that \(1/C(\rho,\delta,L)\) is uniformly bounded, which yields
\[
J_2 \longrightarrow 0 \qquad \text{as } D\to 0 \text{ (equivalently }L\to 0\text{)}.
\]

\textbf{Estimate of $J_3$.} 
Similarly, using the bound $C_i(y_i)\le e^{-\frac12 (|\kappa_i|+\delta)L^2 }$  for any $y_i\in\mathbb{R}$, atogether with $C(\rho,\delta,L)>\frac12 C(\delta)$ for sufficiently small $D$ and $0<\delta<1$, we obtain, for small $D$,
\begin{align}
    J_3&=\int_{\mathbb{R}^N}
a\sum_{i=1}^N
\frac{ C_i(y_i) (|\kappa_i| + \kappa_i + \delta)(|\kappa_i|+\delta)y_i^2}{p_i(y_i)}
\frac{p^2(y)}{C(\rho,\delta,L)}\mathrm{d}y\nonumber\\
&=\sum_{i=1}^N\frac{a(|\kappa_i| + \kappa_i + \delta)(|\kappa_i|+\delta)}{C(\rho,\delta,L)}\int_{\mathbb{R}^N}
\frac{ C_i(y_i) y_i^2}{p_i(y_i)}
 p^2(y)
 \mathrm{d}y\nonumber\\
 &\le\sum_{i=1}^N\frac{a(|\kappa_i| + \kappa_i + \delta)(|\kappa_i|+\delta)e^{-\frac12 (|\kappa_i|+\delta)L^2 }}{C(\rho,\delta,L)}\int_{\mathbb{R}^N}
\frac{ y_i^2}{p_i(y_i)}
 p^2(y)
 \mathrm{d}y\nonumber\\
 &=\sum_{i=1}^N\frac{a(|\kappa_i| + \kappa_i + \delta)(|\kappa_i|+\delta)e^{-\frac12 (|\kappa_i|+\delta)L^2 }}{C(\rho,\delta,L)}\left(\int_{\mathbb{R}} y_i^2 p_i(y_i)\mathrm{d}y_i
\prod_{j\neq i}\int_{\mathbb{R}} p_j^2(y_j) \mathrm{d}y_j\right)\nonumber\\
& =\sum_{i=1}^N\left(
\frac{a (|\kappa_i|+\kappa_i+\delta)(|\kappa_i|+\delta) }{C(\rho,\delta,L)}
\frac{1}{|\kappa_i|+\delta}\sqrt{\frac{2\pi}{|\kappa_i|+\delta}}\cdot\prod_{j\neq i}\sqrt{\frac{\pi}{|\kappa_j|+\delta}} \right)
e^{-\frac{1}{2}(|\kappa_i|+\delta)L^2}\nonumber\\
& \leq
\frac{C}{C(\delta)}\sum_{i=1}^N
e^{-\frac{1}{2}(|\kappa_i|+\delta)L^2},
\end{align}
where $C>0$ is a constant depending on $N$, $a$, $\delta$, and the $\kappa_i$'s. Since $C(\delta)$ is bounded, we obtain $J_3 \longrightarrow 0$ as $D\to0$ (equivalently $L\to\infty$).

\textbf{Estimate of $J_4$.}
Using $\rho=\sqrt{2D/a}$ and $L=D^{-1/8}$, we get
\begin{align*}
    |J_4|    
&\le \frac{C(L^4 D + L^3 D^{\frac{1}{2}}) }{C(\rho,\delta,L)}\sum_{i=1}^N \left(\int_{\mathbb{R}^N}p^2(y)\mathrm{d}y+\int_{\mathbb{R}^N}\frac{p^2(y)}{p_i(y_i)}\mathrm{d}y\right)\nonumber\\
&= C (D^\frac{1}{2} + D^\frac{1}{8}) \frac{C_1(\delta,L)}{C(\rho,\delta,L)} 
 + \frac{C D^\frac{1}{8}}{C(\rho,\delta,L)} \sum_{i=1}^N \left( \int_{\mathbb{R}} p_i(y_i) \, dy_i \right) \left( \prod_{j \neq i} \int_{\mathbb{R}} p_j^2(y_j) \, dy_j \right) \nonumber\\
&\le C D^\frac{1}{8},
\end{align*}
where the last inequality follows from the estimates for $C_1(\delta,L)$ and $C(\rho,\delta,L)$ above, together with the limit
\[
\left( \int_{\mathbb{R}} p_i(y_i) \, dy_i \right) \left( \prod_{j \neq i} \int_{\mathbb{R}} p_j^2(y_j) \, dy_j \right) \to \sqrt{2} \, C(\delta) \quad \text{as } L \to \infty.
\]
Consequently, it follows from the preceding estimate that $J_4\to 0$ as $D\to 0$.

\textbf{Estimate of $J_5$.}
Recall that $\varphi({p^*})=0$, so $\hat{c}(0)=c(\varphi^{-1}(0))=c(p^*)$. From the preceding construction, we have chosen \(0<r_0<\tau\) with \(B(0,r_0)\subset\varphi(U_{p^*})\), and for sufficiently small \(D>0\), \(\sqrt N\rho L<r_0\), so that \(\rho\operatorname{supp}(p(y))\subset B(0,r_0)\). Consequently, for each \(y\in\operatorname{supp}(p(y))\), Taylor's theorem gives
$\hat{c}(\rho y)=\hat{c}(0)+\rho\nabla\hat{c}(0)\cdot y+O(\rho^2|y|^2)$. Since \(p^2(y)\) is even in each variable, we have
\[
\int_{\mathbb{R}^N}
\nabla\hat{c}(0)\cdot y\,p^2(y)\,dy
=0.
\]
In fact,  for every \(i=1,\ldots,N\), we obtain
\begin{align*}
\int_{\mathbb{R}^N}y_i p^2(y)\,dy
=
\left(
\int_{-L}^{L}y_i p_i^2(y_i)\,dy_i
\right)
\prod_{\substack{j=1\\ j\neq i}}^N
\left(
\int_{-L}^{L}p_j^2(y_j)\,dy_j
\right)=0,
\end{align*}
by the oddness of \(y_i p_i^2(y_i)\).
Therefore, a direct computation yields
\begin{align*}
J_5
&= \int_{\mathbb{R}^N} \hat{c}(\rho y) \frac{p^2(y)}{C(\rho,\delta,L)} \, dy \\
&= \int_{\mathbb{R}^N} \left( \hat{c}(0) + \rho\nabla\hat{c}(0)\cdot y+O(\rho^2 |y|^2) \right) \frac{p^2(y)}{C(\rho,\delta,L)} \, dy \\
&= \hat{c}(0) \frac{C_1(\delta,L)}{C(\rho,\delta,L)} + O(\rho^2 L^2) \\
&= \hat{c}(0) \frac{C_1(\delta,L)}{C(\rho,\delta,L)} + O(D^{3/4}) \longrightarrow \hat{c}(0) = c(p^*) \quad \text{as } D \to 0,
\end{align*}
where we use the fact that  $p^2(y)$ is supported in $|y|\le L\sqrt N$ and the convergence follows from   \(C_1(\delta,L)/C(\rho,\delta,L)\to 1\)   as \(D\to0\).

Finally, sending $D \to 0$ (i.e.,, $L \to \infty$ and $\rho \to 0$), and then taking $\delta \to 0$, the estimates for $J_1$ through $J_5$, in particular \eqref{J1}, yield
\[
L_1 + L_5 \le J_1 + J_2 + J_3 + J_4 + J_5
\le c(p^*) + \frac{a}{2} \sum_{i=1}^N \bigl(|\kappa_i(p^*)| + \kappa_i(p^*)\bigr)
\]
for each $p^* \in \Sigma$. Combining this with \eqref{UBL} and sending $\varepsilon \to 0$, we obtain
\[
\limsup_{D \to 0} \lambda(D)
\le \min_{p \in \Sigma} \left\{ c(p) + \frac{a}{2} \sum_{i=1}^N \bigl(|\kappa_i(p)| + \kappa_i(p)\bigr) \right\},
\]
which completes the proof.
\end{proof}

\subsection{Lower bound estimate}
 We first give the local  lower bound estimate for the functional $E[W]$.
\begin{lemma}\label{lem:2.3}
Let \(p_0 \in M\) and let \(U_R:=B_g(p_0,R)\subset M\) be a geodesic ball contained in a normal neighborhood of \(p_0\), with \(R>0\) sufficiently small. If \(W\in H_0^1(U_R)\) (i.e.,  $W=0$ outside $U_R$), then there exists a constant \(C>0\), depending only on \(f\), \(g\), \(p_0\), and \(N\), such that
\begin{equation}\label{I:llb}
E(W) \geq \frac{a}{2} \left( \sum_{i=1}^N \left( |\kappa_i(p_0)| + \kappa_i(p_0) \right) - C R \right) \int_M W^2 \V.
\end{equation}
\end{lemma}

\begin{proof}
Let \((U_R,\varphi)\) be the geodesic normal coordinate chart centered at \(p_0\), as constructed in Section \ref{pre}, and write \(\hat f:=f\circ\varphi^{-1}\). Recall that \(\kappa_i(p_0)\), \(i=1,\dots,N\), denote the eigenvalues of the Hessian operator \(H_{p_0}(f)\); for simplicity we write \(\kappa_i:=\kappa_i(p_0)\) in the following. In these coordinates, \(\varphi(p_0)=0\), and we have the standard identities
\[
g_{ij}(0)=\delta_{ij},\qquad \partial_k g_{ij}(0)=0,\qquad \sqrt{g(0)}=1,
\]
as well as
\begin{equation}\label{hessij-local}
\frac{\partial^2 \hat f}{\partial x_i \partial x_j}(0)
=
H_{p_0}(f)(e_i,e_j)
=
\kappa_i \delta_{ij}.
\end{equation}

Now we express \(E(W)\) in local coordinates and denote $\hat W:=W\circ\varphi^{-1}$
\begin{align*}
 & E (W) \\
    = & 
\int_{U_R} D \left| \nabla_M W - \frac{a}{2D} W \nabla_M f
  \right|^2 \V  \nonumber\\
 =& \int_{\varphi(U_R)}  D g^{ij}(x)(\frac{\partial \hat{W}}{\partial x_i} -\frac{a}{2 D} \hat{W}\frac{\partial \hat{f}}{\partial x_i} )(\frac{\partial \hat{W}}{\partial x_j} -\frac{a}{2 D} \hat{W}\frac{\partial \hat{f}}{\partial x_j} )
 \sqrt{g(x)}\mathrm{d} x\nonumber\\
 = &  
 \int_{\varphi(U_R)} D \sqrt{g(0)}  (g^{i j} (x) - g^{i j} (0)) 
  \left( \frac{\partial \hat{W}}{\partial x_i} - \frac{a}{2D} \hat{W} \frac{\partial
  \hat{f}}{\partial x_i} \right)  \left( \frac{\partial \hat{W}}{\partial x_j} -  \frac{a}{2D} \hat{W} \frac{\partial \hat{f}}{\partial x_j} \right)  \mathrm{d} x
  \nonumber\\
  & + \int_{\varphi(U_R)} D \left( \sqrt{g(x)}  - \sqrt{g(0)}  \right) g^{i j} (x) \left( \frac{\partial \hat{W}}{\partial x_i} -
  \frac{a}{2D} \hat{W} \frac{\partial \hat{f}}{\partial x_i} \right)  \left(
  \frac{\partial \hat{W}}{\partial x_j} - \frac{a}{2D} \hat{W} \frac{\partial
  \hat{f}}{\partial x_j} \right)  \mathrm{d} x \nonumber\\
  &+  \int_{\varphi(U_R)} D  \left| \nabla \hat{W} - \frac{a}{2D} \hat{W} \nabla \hat{f}   \right|^2 \mathrm{d} x  \nonumber\\ 
=: &I_1+I_2+I_3. 
\end{align*}
where
\begin{align*}
I_1 &:=
 \int_{\varphi(U_R)} D \sqrt{g(0)}  (g^{i j} (x) - g^{i j} (0)) 
  \left( \frac{\partial \hat{W}}{\partial x_i} - \frac{a}{2D} \hat{W} \frac{\partial
  \hat{f}}{\partial x_i} \right)  \left( \frac{\partial \hat{W}}{\partial x_j} -  \frac{a}{2D} \hat{W} \frac{\partial \hat{f}}{\partial x_j} \right)  \mathrm{d} x \nonumber\\
I_2 &:=
\int_{\varphi(U_R)} D \left( \sqrt{g(x)}  - \sqrt{g(0)}  \right) g^{i j} (x) \left( \frac{\partial \hat{W}}{\partial x_i} -
\frac{a}{2D} \hat{W} \frac{\partial \hat{f}}{\partial x_i} \right) 
\left(\frac{\partial \hat{W}}{\partial x_j} - \frac{a}{2D} \hat{W}\frac{\partial\hat{f}}{\partial x_j} \right)  \mathrm{d} x \nonumber,\\
I_3 &:=
\int_{\varphi(U_R)} D  \left| \nabla \hat{W} - \frac{a}{2D} \hat{W} \nabla \hat{f}   \right|^2 \mathrm{d} x  .
\end{align*}
In geodesic normal coordinates, we recall the metric expansions \eqref{tay}. Consequently, there exists a constant \(C>0\), independent of \(R\) and \(D\), such that
\[
\sup_{x\in B_R(0)}
\left( |g^{ij}(x)-\delta^{ij}| + |\sqrt{ g(x)}-1| \right) \le C R^2.
\]
Hence,
\[
|I_1| + |I_2| \le C R^2 I_3.
\]
Taking \(R\) sufficiently small so that \(CR^2<1\), we obtain
\[
E(W) \ge (1 - C R^2) I_3.
\]

It remains to estimate \(I_3\). Applying Lemma 2.3(i) in \cite{PZZ19} to \(\hat W\) and \(\hat f\), using \eqref{hessij-local}, we obtain
\begin{align}
    I_3\geq
    \frac{a}{2}
    \left[
        \sum_{i=1}^N\bigl(|\kappa_i|+\kappa_i\bigr)
        -
        2NR
        \|D^3\hat f\|_{L^\infty(B(0,R))}
    \right]
    \int_{B(0,R)}\hat W^2\,\mathrm dx.
    \label{I3}
\end{align}

On the other hand, from the metric expansion \eqref{tay}, there exists a constant \(C>0\) such that
\[
    1-CR^2
    \leq
    \sqrt{g(x)}
    \leq
    1+CR^2
    \qquad
    \text{in }B(0,R).
\]
Consequently, the Euclidean and Riemannian \(L^2\)-norms satisfy
\[
    (1-CR^2)
\int_{B(0,R)}\hat W^2\,\mathrm dx
    \leq\!
\int_{U_R}W^2\,\mathrm dV_g
=\! \int_{B(0,R)}\hat{W}^2\sqrt{g(x)}\,\mathrm dx
    \leq
    (1+CR^2) \!   \int_{B(0,R)}\hat W^2\,\mathrm dx.
\]
Since \(M\) is compact and \(f\) is smooth, the third-order derivatives of \(\hat f\) are uniformly bounded in \(B(0,R)\) for sufficiently small \(R\). Hence there exists a constant \(C>0\), independent of \(R\), such that
\[
\|D^3\hat f\|_{L^\infty(B(0,R))} \le C.
\]Combining this with the preceding estimates and enlarging the constant \(C\) if necessary, we conclude that
\begin{equation}\label{eq:manifold-local-lower-bound}
E(W)
\geq
\frac{a}{2}
\left(
\sum_{i=1}^N (|\kappa_i|+\kappa_i)
-
CR
\right)
\int_M W^2 \V,
\end{equation}
where \(C>0\) depends only on \(f\), \(g\), \(p_0\), and \(N\). This proves the lemma.
\end{proof}

 Based on Lemma \ref{lem:2.3}, we have the following lower bound estimate:
\begin{proposition}\label{prop:lower} 
For the principal eigenvalue $\lambda(D)$, we have the following estimate:
\[
\liminf_{D \to 0} \lambda(D) \geq \min_{p \in \Sigma} \left\{ c(p) + \frac{a}{2} \sum_{i=1}^N \left( |\kappa_i(p)| + \kappa_i(p) \right) \right\}.
\]
\end{proposition}

\begin{proof}
Recall that the critical set of the Morse function \(f\) on \(M\) is finite. Write
\[
\Sigma=\{p_1,\ldots,p_K\}, \qquad K\in\mathbb N.
\]
Let \(w\) be the principal eigenfunction normalized by \(\int_M w^2 \V=1\). Choose \(R>0\) sufficiently small such that the geodesic balls \(B_g(p_k,R)\) are pairwise disjoint:
\[
B_g(p_k,R)\cap B_g(p_l,R)=\emptyset \qquad (k\neq l).
\]
Set
\[
M_0:=M\setminus \bigcup_{k=1}^K \overline{B_g(p_k,R/2)} .
\]
Then
\[
\{B_g(p_1,R),\ldots,B_g(p_K,R),M_0\}
\]
is an open covering of \(M\).

Let
$\{\zeta_k^2\}_{k=0}^K$
be a smooth partition of unity subordinate to this covering. Thus
\(\zeta_k\in C^\infty(M)\), \(0\leq \zeta_k\leq 1\), and
\[
\sum_{k=0}^K \zeta_k^2(p)=1,
\qquad \forall p\in M,
\]
with
\[
\operatorname{supp} \zeta_0\subset M_0,
\qquad
\operatorname{supp} \zeta_k\subset B_g(p_k,R),
\quad (k=1,\ldots,K).
\]
Moreover, the functions can be chosen so that
\[
|\nabla_M\zeta_k|\leq \frac{C}{R},
\qquad k=0,1,\ldots,K,
\]
for some constant \(C>0\) independent of \(R\).  Since the balls \(B_g(p_k,R)\) are pairwise disjoint, each point \(p\in M\) belongs to at most two sets in the covering; consequently, at most two of the functions \(\zeta_k\) are nonzero at any point. Hence,
\begin{equation}\label{gradient bound}
\sum_{k=0}^K|\nabla_M\zeta_k(p)|^2
\leq\frac{C}{R^2},
\qquad \forall p\in M.
\end{equation}

Define \(W_k = \zeta_k w\) for \(k = 0, 1, \ldots, K\). Then $W_k\in H_0^1(B_g(p_k,R))$ for \(1\leq k\le K\).
Thanks to Lemma \ref{lem:2.3}, we obtain
\[
E(W_k) \geq \frac{a}{2} \left( \sum_{i=1}^N \left( |\kappa_i(p_k)| + \kappa_i(p_k) \right) - C R \right) \int_M W_k^2 \V.
\]
From Lemma \ref{lem:2.2}, for each \(1\le k\le K\),
\begin{align*}
   & \lambda(D) \int_M W_k^2 \V + D \int_M \frac{W_k^2}{\zeta_k^2} |\grad \zeta_k|^2 \V\\=&
E(W_k) +\int_M    c W_k^2   \V \\
\ge&
   \frac{a}{2} \left( \sum_{i=1}^N \left( |\kappa_i(p_k)| + \kappa_i(p_k) \right) - C R \right) \int_M W_k^2 \V
   +\int_M  c W_k^2 \V. 
\end{align*}
Summing over $k=1,...,K$, we get
\begin{align}
&\lambda(D) \sum_{k=1}^K \int_M W_k^2 \V \label{I:lb} \\
\geq& \left( \min_{1 \leq k \leq K} \left\{ c(p_k) + \frac{a}{2} \sum_{i=1}^N (|\kappa_i(p_k)| + \kappa_i(p_k)) \right\}\right)\sum_{k=1}^K \int_M W_k^2 \V  - C R  - D \cdot \frac{C}{R^2}.\nonumber
\end{align}
where the term \(-CR\) comes from Lemma \ref{lem:2.3} and the replacement of \(c(p_k)\) by its local minimum with the estimate \(\sum_{k=1}^K\int_M W_k^2\V\le \int_M w^2\V=1\), while the term \(-C D/R^2\) follows from Lemma \ref{lem:2.2} together with the gradient bound \eqref{gradient bound}

It remains to estimate the contribution from the non-critical region \(M_0\). Since \(M_0\) is compact and contains no critical points of \(f\), there exists \(\delta>0\) such that \(|\grad f| \geq \delta\). Applying Lemma \ref{lem:2.1} to \(W_0\), we obtain
\begin{align*}
E(W_0)\ge &\frac{a^2}{4D}  \int_{M} \left[ | \nabla_M f|^2 +
     \frac{2D}{a} \Delta_M f \right] W_0^2 \mathrm{d} \mathrm{V_g}\\
     =&\frac{a^2}{4D}  \int_{M_0} \left[ | \nabla_M f|^2 +
     \frac{2D}{a} \Delta_M f \right] W_0^2 \mathrm{d} \mathrm{V_g}\\
     \ge& \frac{a^2}{4D}  \int_{M_0} \left[ \delta^2 -
     \frac{2D}{a} \|\Delta_M f \|_{L^\infty(M)}\right] W_0^2 \mathrm{d} \mathrm{V_g}.
\end{align*}
For $D>0$ sufficiently small such that 
\(\frac{2D}{a}\|\Delta_M f\|_{L^\infty(M)} \le \delta^2/2\), we have
\begin{align*}
\int_{\Omega_0}  W_0^2 \mathrm{d} \mathrm{V_g}\leq
\frac{4D}{a^2\left[ \delta^2 -
     \frac{2D}{a} |\Delta_M f |_{L^\infty}\right]}E(W_0)\leq \frac{8 D}{\delta^2 a^2}E(W_0).
\end{align*}
On the other hand, from Lemma \ref{lem:2.2} and the upper bound for $\lambda(D)$ above (Proposition \ref{prop:upper}), we have 
\begin{align*}
E(W_0)\leq |\lambda(D)|\int_M w^2\V+\max_{p\in M}c(p)\int_M w^2\V+C_R D\int_M w^2\V \leq C_R
\end{align*}
for some constant \(C_R>0\) depending on \(R\) (and on \(f,g\)) but independent of \(D\). Hence,    
\[
\int_Mw^2 \sum_{i=1}^K \zeta_k^2\V=1-\int_M W_0^2 \V\ge 1-C_R D.
\]
Letting \(D\to0\) (with \(R>0\) fixed), we obtain
\[
\int_Mw^2 \sum_{i=1}^K \zeta_k^2\V=\sum_{k=1}^K \int_M W_k^2 \V \to 1.
\]
Taking \(\liminf\) as \(D \to 0\) in \eqref{I:lb} and then letting \(R \to 0\), we conclude
\[
\liminf_{D \to 0} \lambda(D) \geq \min_{p \in \Sigma} \left\{ c(p) + \frac{a}{2} \sum_{i=1}^N (|\kappa_i(p)| + \kappa_i(p)) \right\},
\]
which completes the proof.  
\end{proof}

\noindent\textbf{Conflict of interest:} The authors declare that they have no conflict of interest.

\noindent\textbf{Acknowledgements:}
 The research of Xin Xu is supported by NSF of
China (Youth Program No. 12401255).
\bibliographystyle{abbrv}
\bibliography{bib}

\end{document}